\documentclass{amsart}

\usepackage{amsmath}
\usepackage{amssymb}
\usepackage{amsthm}
\usepackage{amscd}

\newtheorem{thm}{Theorem}[section]
\newtheorem{prop}[thm]{Proposition}
\newtheorem{lem}[thm]{Lemma}

\theoremstyle{definition}
\newtheorem{dfn}[thm]{Definition}

\theoremstyle{remark}
\newtheorem{rem}{Remark}

\newcommand{\C}{\mathbb{C}}

\newcommand{\Z}{\mathbb{Z}}

\newcommand{\ch}{\mathrm{ch}}

\allowdisplaybreaks[4]

\title{A basis of the Atiyah-Segal invariant polynomials}

\author{Kiyonori Gomi}

\address{
Department of Mathematics, 
Kyoto University, 
Kyoto, 606-8502, 
Japan.
}

\email{kgomi@math.kyoto-u.ac.jp}

\date{}

\begin{document}

\begin{abstract}
For twisted $K$-theory whose twist is classified by a degree three integral cohomology of infinite order, universal even degree characteristic classes are in one to one correspondence with invariant polynomials of Atiyah and Segal. The present paper describes the ring of these invariant polynomials by a basis and structure constants. 
\end{abstract}

\maketitle



\section{Introduction} 
\label{sec:introduction}


\subsection{Atiyah-Segal invariant polynomials}

The ring of Atiyah-Segal invariant polynomials \cite{A-Se2} is a subring in the polynomial ring $A = \C[x_1, x_2, \cdots]$ on generators $x_i$ of degree $i$. Its definition \footnote{In \cite{A-Se2}, the ring is first defined by using the variables $s_i = i!x_i$.} is $J = \mathrm{Ker}(d)$, the kernel of the derivation $d : A \to A$ given by 
\begin{align*}
dx_1 &= 0, &
dx_i &= x_{i-1}, \ (i > 1).
\end{align*}
Originally, $J$ is introduced as the ring of the universal Chern classes of \textit{twisted $K$-theory} \cite{A-Se1} whose twist is classified by a degree three integral cohomology class of infinite order. While more study is dedicated to twisted $K$-theory in recent years, few is known about its Chern classes and $J$. For example, $J$ is not isomorphic to a polynomial ring \cite{A-Se2}, whereas there seems to remain the issue of presenting $J$ by generators and relations.


\subsection{Main result}

The aim of this paper is to describe the ring structure of $J$ by a basis and structure constants. 
To state the description precisely, we define, for an integer $\ell \ge 0$, a set $B^{(\ell)}$ by
$$
B^{(\ell)}
=
\left\{
\beta = (\beta_1, \beta_2, \cdots, \beta_\ell) \in \Z^\ell \big| \
\begin{array}{c}
\beta_1, \cdots, \beta_{\ell - 1} \ge 0, \ \beta_{\ell} \ge 1
\end{array}
\right\}, \ (\ell \ge 1),
$$
and $B^{(0)} = \{ \beta = (\emptyset) \}$. We also define $B^{(\ell)}(0)$ by the following subset in $B^{(\ell)}$:
\begin{align*}
B^{(0)}(0) &= B^{(0)}, &
B^{(1)}(0) &= \{ \beta = (1) \}, &
B^{(\ell)}(0) &= \{ \beta \in B^{(\ell)} |\ \beta_1 = 0 \}, \ (\ell > 1).
\end{align*}
For $\beta \in B^{(\ell)}$, $\beta' \in B^{(\ell')}$ and $\beta'' \in B^{(\ell + \ell')}$, we define the number $N_{\beta \beta'}^{\beta''}$ by the following expression of a polynomial:
\begin{multline*}
e_1(\vec{k}, \vec{k}')^{\beta_1''} \cdots 
e_{\ell+\ell'}(\vec{k}, \vec{k}')^{\beta''_{\ell+\ell'}} \\
=
\sum_{\beta \in B^{(\ell)}, \beta' \in B^{(\ell')}}
N_{\beta \beta'}^{\beta''}
e_1(\vec{k})^{\beta_1} \cdots e_\ell(\vec{k})^{\beta_\ell}
e_1(\vec{k}')^{\beta'_1} \cdots e_{\ell'}(\vec{k}')^{\beta'_{\ell'}},
\end{multline*}
where $e_i(\vec{k})$, $e_i(\vec{k}')$ and $e_i(\vec{k}, \vec{k}')$ mean the $i$-th elementary symmetric polynomials in $\{ k_1, \cdots, k_\ell \}$, $\{ k'_1, \cdots, k'_{\ell'}\}$ and $\{ k_1, \cdots, k_\ell, k'_1, \cdots, k'_{\ell'} \}$, respectively. Note that $N_{\beta \beta'}^{\beta''}$ are non-negative integers, because $e_i(\vec{k}, \vec{k}') = \sum_{j}e_j(\vec{k})e_{i-j}(\vec{k}')$.

\begin{thm} \label{thm:main}
The ring $J$ of the Atiyah-Segal invariant polynomials is isomorphic to the ring $J'$ constructed as follows: 
\begin{itemize}
\item[(a)]
its underlying vector space over $\C$ is generated by $\beta \in B(0) = \bigcup_{\ell \ge 0}B^{(\ell)}(0)$; 

\item[(b)]
the product of $\beta \in B^{(\ell)}(0)$ and $\beta' \in B^{(\ell')}(0)$ is given by
$$
\beta \cdot {\beta'} 
= \sum_{\beta'' \in B^{(\ell+\ell')}(0)}
N_{\beta \beta'}^{\beta''} {\beta''}.
$$
\end{itemize}
\end{thm}

Denote by $g_\beta(\vec{x})$ the invariant polynomial corresponding to $\beta$. Then $g_{(\emptyset)}(\vec{x}) = 1$ and $g_{(1)} (\vec{x})= x_1$. Some explicit product formulae of the polynomials are:
\begin{align*}
g_{(1)} g_{(1)} &= g_{(0, 1)}, &
g_{(1)} g_{(0, \beta_2, \cdots, \beta_{\ell-1}, \beta_\ell)}
&=
g_{(0, \beta_2, \cdots, \beta_{\ell-1}, -1 + \beta_\ell, 1)},
\end{align*}
\begin{align*}
g_{(0, a)} g_{(0, b)}
&=
\sum_{1 \le r \le \frac{a + b}{2}}
\frac{(a + b - 2r)!}{(a - r)!(b - r)!}
g_{(0, a + b - 2r, 0, r)}, \\
g_{(0, a)} g_{(0, b, c)}
&=
\sum_{
\substack{
0 \le r \le \mathrm{min}\{a-s, b\} \\
1 \le s \le c \\
}}
\frac{(a + b - 2r - s)!}{(a - r - s)!(b - r)!}
g_{(0, a + b - 2r - s, c - s, r, s)}.
\end{align*}


\subsection{Application}

In \cite{A-Se2}, the Chern classes of twisted $K$-theory corresponding to $J$ live in ordinary cohomology, at the first stage. Then lifting these Chern classes to twisted cohomology is proposed, and the problem to find some natural basis is raised. We answer this problem by constructing a lift of the basis in Theorem \ref{thm:main}


\subsection{Toward a description of $J$ by generators and relations}

If we consider the polynomial algebra generated on $B(0)$ and take the quotient by the ideal generated by relations corresponding to (b) in Theorem \ref{thm:main}, then we get a description of $J$ by generators and relations, which is obviously unsatisfactory. For a satisfactory description, a possible direction would be to eliminate redundant bases. With respect to an order of monomials in $x_i$, the leading term of the invariant polynomial $g_\beta(\vec{x}) $ corresponding to $\beta = (\beta_1, \cdots, \beta_\ell) \in B^{(\ell)}(0)$ is the monomial $x_{\beta_\ell}x_{\beta_\ell + \beta_{\ell-1}} \cdots x_{\beta_\ell + \cdots + \beta_1}$, and its coefficient is always $1$. This fact leads us to introduce the subset of $\{ g_\beta |\ \beta \in B(0) \}$ consisting of bases $g_\beta(\vec{x})$ whose leading terms cannot be the products of the leading terms of other non-trivial bases of strictly lower degree. We can see that the bases in the subset generate the ring $J$ algebraically, and their first non-trivial relations appear in degree $12$:
\begin{align*}
g_{(0,2)}g_{(0,1,2)} - g_{(0,3)} g_{(0,0,2)}
&=
g_{(1)} g_{(0,0,1,2)} + g_{(1)}^2 g_{(0,2,2)}, \\
g_{(0,2)}^3 - g_{(0,0,2)}^2 
&=
3g_{(1)}^2 g_{(0,2)}g_{(0,3)} - 2g_{(1)}^3g_{(0,0,3)} - 3g_{(1)}^4g_{(0,4)}.
\end{align*}
However, by computing relations in higher degree, we can also see that the subset still contains many redundant bases to generate $J$ algebraically. The task to find out a minimal set of algebraic generators still seems not easy, and the presentation issue of $J$ needs further study.


\subsection{Plan of the paper}

The description of $J$ in Theorem \ref{thm:main} originates from a construction of characteristic classes for twisted $K$-theory based on the \textit{Chern character} and the \textit{Adams operations} \cite{A-Se2}. Section \ref{sec:motivating_construction} is devoted to this motivating construction. Though this section can be skipped for the proof of Theorem \ref{thm:main}, one with the understanding of the construction will find the definition of our basis natural. 
The proof of Theorem \ref{thm:main}, which is quite elementary, is given in Section \ref{sec:basis}. The point in the proof is the definition of the polynomials $g_\beta(\vec{x})$ associated to $\beta \in B = \bigcup_{\ell \ge 0} B^{(\ell)}$. With the additive basis formed by these polynomials, the structure constant of $A$ and the subring $J \subset A$ are easy to express. Our lifts of Chern classes are provided at the end of this section. 

Finally, a list of invariant polynomials $g_\beta(\vec{x})$ is appended for reference.


\subsection{Acknowledgment}

This work is supported by the Grants-in-Aid for Scientific Research (start-up 21840034), JSPS.


\section{Motivating construction}
\label{sec:motivating_construction}


\subsection{Chern character}

One tool for the construction motivating our basis is the \textit{Chern character} for twisted $K$-theory. For simplicity, we assume that $X$ is a smooth manifold. With some subtle technical points understood, the Chern character is a natural homomorphism
$$
\ch : \ K_\tau(X) \longrightarrow H^{\mathrm{even}}_\eta(X).
$$
Here $K_\tau(X)$ denotes the twisted $K$-group of $X$ with its twist $\tau$. Since the way to realize twists may vary according to the contexts, we just point out that a twist is a geometric object classified by the degree three integral cohomology group $H^3(X; \Z)$. For the sake of simplicity, we also assume that the twist $\tau$ is classified by an element in $H^3(X; \Z)$ of infinite order. The target of $\ch$ is the cohomology of the complex $(\Omega(X), d - \eta \wedge \cdot )$, where $\Omega(X)$ is the space of differential forms, and $\eta$ is a closed $3$-form whose de Rham cohomology class corresponds to the real image of the element in $H^3(X; \Z)$ classifying $\tau$.

A point to note is the following commutative diagram:
$$
\begin{CD}
K_\tau(X) \times K_{\tau'}(X) @>{\otimes}>> K_{\tau + \tau'}(X) \\
@V{\ch \times \ch}VV @VV{\ch}V \\
H_{\eta}(X) \times H_{\eta'}(X) @>>{\wedge}> H_{\eta + \eta'}(X),
\end{CD}
$$
where $\otimes$ is the multiplication in twisted $K$-theory, which mixes twists. 

From the Chern character, we can compute the Chern class corresponding to an invariant polynomial $f(x_1, x_2, \cdots) \in J$: Let $a \in K_\tau(X)$ be a twisted $K$-class. We express its Chern character as follows:
$$
\ch(a) = [x_1(a) + x_2(a) + x_3(a) + \cdots],
$$
where $x_n(a)$ is the $2n$-form part. (The $0$-form part $x_0(a)$ is absent, under the assumption on $\tau$.) The differential form $f(x_1(a), x_2(a), \cdots)$ is closed by the invariance of the polynomial $f$. If we denote the de Rham cohomology class of the differential form by $c_f(a)$, then we get the Chern class corresponding to $f$:
$$
c_f : \ K_\tau(X) \longrightarrow  H^{\mathrm{even}}(X).
$$


\subsection{Adams operation}

The other tool for our motivating construction is the \textit{Adams operation} for twisted $K$-theory \cite{A-Se2}, which is given as a natural map
$$
\psi^k : K_\tau(X) \longrightarrow K_{k\tau}(X).
$$
Here $k$ is allowed to be any integer. The twists in the source and the target of $\psi^k$ are generally different, because the formulation of $\psi^k$ involves the multiplication in twisted $K$-theory. It should be noticed that, under the expression $\ch(a) = [x_1(a) + x_2(a) + \cdots]$, we can express the Chern character of $\psi^k(a)$ as
$$
\ch(\psi^k(a)) = [k x_1(a) + k^2 x_2(a) + k^3 x_3(a) + \cdots].
$$


\subsection{Factory of Chern classes}

For integers $k_1, \cdots, k_\ell$, the Adams operations $\psi^{k_i}$ and the product in twisted $K$-theory construct
$$
\begin{array}{rcl}
\psi^{(k_1, \cdots, k_\ell)} : \
K_\tau(X) & \longrightarrow &
K_{(k_1 + \cdots + k_\ell)\tau}(X). \\
 & &  \\
a \ \quad & \mapsto & \psi^{k_1}(a) \otimes \cdots \otimes \psi^{k_\ell}(a)
\end{array}
$$
While Atiyah and Segal considered the case of $k_1 + \cdots + k_\ell =  1$ to get an ``internal'' operation, we consider the case of $k_1 + \cdots + k_\ell = 0$. The resulting ``external'' operation followed by the Chern character then gives
$$
\ch \circ \psi^{(k_1, \cdots, k_\ell)} : \
K_\tau(X) \longrightarrow H^{\mathrm{even}}(X).
$$
Since this map is natural, its $2n$-form part gives rise to a Chern class of twisted $K$-theory. It is natural to ask what kind of Chern classes are produced by this method: These Chern classes correspond to some polynomials in $J$. 

By means of the properties explained so far, we can compute examples readily. In the simplest case of $\ell = 2$, we have
\begin{align*}
\mathrm{ch}(\psi^{(k_1,k_2)}(a))
&=
(k_1k_2) \cdot x_1^2 \\
&+ 
(k_1k_2)^2 \cdot (x_2^2 - 2 x_1x_3) \\
&+
(k_1k_2)^3 \cdot (x_3^2 - 2 x_2x_4 + 2 x_1x_5) \\
&+
(k_1k_2)^4 \cdot 
(x_4^2 - 2 x_3x_5 + 2 x_2x_6 - 2x_1x_7) \\
&+
(k_1k_2)^5 \cdot (x_5^2 - 2 x_4x_6 + 2 x_3x_7 - 2x_2x_8 + 2 x_1x_9) \\
&+ 
\cdots, 
\end{align*}
where $x_i = x_i(a)$.
In the case of $\ell = 3$, we also have
\begin{align*}
\mathrm{ch}(\psi^{(k_1,k_2,k_3)}(a))
&=
e_3 \cdot  x_1^3 \\
&+
e_2e_3  \cdot x_1(x_2^2 - 2x_1x_3) \\
&+
e_3^2 \cdot  (x_2^3 - 3 x_1x_2x_3 + 3 x_1^2x_4) \\
&+
e_2^2 e_3 \cdot x_1(x_3^2 - 2x_2x_4 + 2x_1x_5) \\
&+
e_2e_3^2 \cdot 
(x_2x_3^2 - 2 x_2^2x_4 - x_1x_3x_4 + 5 x_1x_2x_5 - 5 x_1^2x_6) \\
&+
e_3^3 \cdot 
(x_3^3 - 3 x_2x_3x_4 + 3 x_2^2x_5 \\
& \quad \quad \quad \quad 
+ 3 x_1x_4^2 - 3 x_1x_3x_5 - 3 x_1x_2x_6 + 3 x_1^2x_7) \\
&\quad + 
e_2^3e_3 \cdot 
x_1(x_4^2 - 2 x_3x_5 + 2 x_2x_6 - 2 x_1x_7) \\
&+ 
\cdots,
\end{align*}
where $e_2 = k_1k_2 + k_2k_3 + k_1k_3$ and  $e_3 = k_1k_2k_3$ for short.

With some experience of calculating polynomials in $J$, one will find that each coefficient of a product of elementary symmetric polynomials in $k_1, \cdots, k_\ell$ is an invariant polynomial. Further, one may guess that the invariant polynomials arising in this way constitute an additive basis of a subspace in $J$: This turns out to be the case, as a result of Theorem \ref{thm:main}. In the next section, we consider $\ch \circ \psi^{(k_1, \cdots, k_\ell)}$ in purely algebraic setting to construct our basis of $J$.


\section{A basis of invariant polynomials}
\label{sec:basis}


\subsection{Preliminary}

We define $\ch(k | \vec{x})$ to be the following formal power series in variables $x_1, x_2, \cdots$ and $k$:
$$
\ch(k | \vec{x}) =
\sum_{i \ge 1} k^i x_i =
k x_1 + k^2 x_2 + k^3 x_3 + \cdots.
$$
For an integer $\ell \ge 1$, we let $\ch(k_1, \cdots, k_\ell | \vec{x})$ be the following formal power series in the variables $x_1, x_2, \cdots$ and $k_1, \cdots, k_\ell$:
$$
\ch(k_1, \cdots, k_\ell | \vec{x})
=
\ch(k_1 | \vec{x}) \cdots \ch(k_\ell | \vec{x}) 
=
\sum_{i_1, \cdots, i_\ell \ge 1} 
k_1^{i_1} \cdots k_r^{i_\ell} x_{i_1} \cdots x_{i_\ell}.
$$
We write $\ch(k_1, \cdots, k_\ell | \vec{x})_n$ for the degree $n$ part of $\ch(k_1, \cdots, k_\ell | \vec{x})$ with respect to $x_i$. By construction, $\ch(k_1, \cdots, k_\ell | \vec{x})_n$ is a symmetric polynomial in $k_1, \cdots k_\ell$ with its coefficients in $A$. In particular, the symmetric polynomial is of degree $n$, provided that each $k_i$ is given degree $1$.

As is well-known \cite{M}, the ring of symmetric polynomials in $k_1, \cdots, k_\ell$ is the polynomial ring in the elementary symmetric polynomials $e_1, \cdots, e_\ell$:
$$
e_j(\vec{k}) = e_j(k_1, \cdots, k_\ell)
=
\sum_{1 \le i_1 < \cdots < i_j \le \ell}
k_{i_1} \cdots k_{i_j}.
$$
For integers $n$ and $\ell$ such that $n \ge \ell \ge 1$, we put
$$
B_{n}^{(\ell)}
=
\left\{
\beta = (\beta_1, \beta_2, \cdots, \beta_\ell) \in \Z^\ell \bigg| \
\begin{array}{c}
\beta_1, \cdots, \beta_{\ell-1} \ge 0, \ \beta_{\ell} \ge 1 \\
\beta_1 + 2 \beta_2 + \cdots + \ell \beta_{\ell} = n
\end{array}
\right\}.
$$
We set $B^{(0)}_0 = B^{(0)}$ and $B^{(0)}_n = \emptyset$ for $n \ge 1$. Then $B^{(\ell)}$ in Section \ref{sec:introduction} is $B^{(\ell)} = \bigcup_{n \ge \ell} B_{n}^{(\ell)}$. To an element $\beta = (\beta_1, \cdots, \beta_\ell) \in B_n^{(\ell)}$, we associate the product of elementary symmetric polynomials $e^\beta(\vec{k}) = e_1(\vec{k})^{\beta_1} \cdots e_\ell(\vec{k})^{\beta_\ell}$. These products form a basis of the vector space of symmetric polynomials of degree $n$ in $k_1, \cdots, k_\ell$.

\begin{dfn}
Let $n$ and $\ell$ be integers such that $n \ge \ell \ge 1$. For $\beta \in B_n^{(\ell)}$, we define a polynomial $g_\beta(\vec{x}) \in A$ of degree $n$ by the following formula:
$$
\ch(k_1, \cdots, k_\ell | \vec{x})_n
=
\sum_{\substack{i_1, \cdots, i_r \ge 1 \\ i_1 + \cdots + i_\ell = n}}
k_1^{i_1} \cdots k_\ell^{i_\ell}
x_{i_1} \cdots x_{i_\ell}
=
\sum_{\beta \in B_n^{(\ell)}}
e^\beta(k_1, \cdots, k_\ell) g_\beta(\vec{x}).
$$
For $\beta = (\emptyset) \in B^{(0)}$, we define $g_\beta(\vec{x}) = 1$.
\end{dfn}

Another equivalent definition of $g_\beta(\vec{x})$ makes use of the transition matrix $M = M(m, e)$ from elementary symmetric polynomials to monomial symmetric polynomials \cite{M}: let $\Lambda_n^{(\ell)}$ be the set of partitions of $n$ of length $\ell$:
$$
\Lambda_n^{(\ell)}
=
\left\{
\lambda = (\lambda_1, \cdots, \lambda_\ell) \in \Z^\ell | \
\lambda_1 \ge \cdots \ge \lambda_\ell \ge 1, \ 
\lambda_1 + \cdots + \lambda_\ell = n
\right\},
$$
which can be identified with $B_n^{(\ell)}$ through the change of expression:
$$
\beta = (\beta_1, \cdots, \beta_\ell) \leftrightarrow
\lambda =
(
\overbrace{\ell, \cdots, \ell}^{\beta_\ell},
\cdots,
\overbrace{2, \cdots, 2}^{\beta_2},
\overbrace{1, \cdots, 1}^{\beta_1}
).
$$
For $\lambda \in \Lambda_n^{(\ell)}$, the \textit{monomial symmetric polynomial} $m_\lambda(\vec{k}) = m_\lambda(k_1, \cdots, k_\ell)$ is the polynomial $\sum k_1^{\mu_1} \cdots k_\ell^{\mu_\ell}$ summed over all distinct permutations $(\mu_1, \cdots, \mu_\ell)$ of $\lambda$. They also form a basis of the space of symmetric polynomials of degree $n$ in $k_i$. Let $M_{\lambda \beta}$ be the transition matrix given by the base change $m_\lambda = \sum_\beta M_{\lambda \beta} e^\beta$. Because of the expression
$$
\ch(k_1, \cdots, k_\ell | \vec{x})_n
=
\sum_{\lambda \in \Lambda_n^{(\ell)}}
m_\lambda(k_1, \cdots, k_\ell) x_\lambda
=
\sum_{\beta \in B_n^{(\ell)}}
e^\beta(k_1, \cdots, k_\ell) g_\beta(\vec{x}),
$$
the other definition of $g_\beta(\vec{x})$ is
$$
g_\beta(\vec{x}) = 
\sum_{\lambda \in \Lambda_n^{(\ell)}}
M_{\lambda \beta} x_\lambda,
$$
where $x_\lambda = x_{\lambda_1} \cdots x_{\lambda_\ell}$ for $\lambda = (\lambda_1, \cdots, \lambda_\ell)$.

\begin{rem}
Since $M_{\lambda \beta} \in \Z$, the latter definition shows $g_\beta(\vec{x}) \in \Z[x_1, x_2, \cdots ] \subset A$.
\end{rem}

\begin{rem}
The transition matrix $(M_{\lambda \beta})_{\lambda \beta}$ can be computed from the \textit{Kostka matrix}. Some facts about the matrix in \cite{M} imply the expression:
$$
g_\beta(\vec{x}) = 
x_{\beta'}
+ \sum_{\substack{\lambda \in \Lambda_n^{(\ell)} \\ \beta' < \lambda}} 
M_{\lambda \beta} x_\lambda.
$$
Here $\beta' \in \Lambda_n$ is the \textit{conjugate} of the partition $\beta$. The meaning of $\beta' < \lambda$ is that $\beta' \neq \lambda$ and $\beta' \le \lambda$, where $\le$ is the \textit{natural (partial) ordering} \cite{M}.
\end{rem}

\begin{rem}
For $n \ge \ell$, let $\omega \in \Lambda_n^{(\ell)}$ be $\omega = (n - \ell + 1, \overbrace{1, \cdots, 1}^{\ell-1})$, which satisfies $\lambda \le \omega$ for all $\lambda \in \Lambda_n^{(\ell)}$. For any $\beta = (\beta_1, \beta_2, \cdots, \beta_\ell) \in B_n^{(\ell)}$, the coefficient $M_{\omega\beta}$ of $x_\omega$ in $g_\beta(\vec{x})$ never vanish: If $n = \ell$, then $M_{\omega\beta} = 1$. If $n > \ell$, then we obtain
$$
M_{\omega \beta}
=
(-1)^{\ell + 1 + \sum_{i}\beta_{2i}}
\frac{n - \ell}{(\sum_i \beta_i) - 1}
\frac{((\sum_i \beta_i) - 1)!}{\prod_i (\beta_i!)} \beta_\ell,
$$
by using so-called \textit{Waring's formula}, an explicit formula expressing the power sum in terms of the elementary symmetric functions (page 33, Example 20, \cite{M}).
\end{rem}


\subsection{Proof of Theorem \ref{thm:main}}

We denote by $A_n^{(\ell)} \subset A$ the subspace consisting of polynomials of degree $n$ in $x_1, x_2, \cdots$ which are linear combinations of monomials $x_{i_1} \cdots x_{i_\ell}$ of length $\ell$. We set $A_n = \bigoplus_\ell A^{(\ell)}_n$ and $A^{(\ell)} = \bigoplus_n A^{(\ell)}_n$. By construction, the polynomial $g_\beta(\vec{x})$ with $\beta \in B_n^{(\ell)}$ belongs to the subspace $A_n^{(\ell)}$.

\begin{lem} \label{lem:polynomial_ring}
The following holds about the polynomial ring $A = \C[x_1, x_2, \cdots]$:
\begin{itemize}
\item[(a)]
For $n \ge \ell \ge 0$, the set $\{ g_\beta(\vec{x}) |\ \beta \in B_n^{(\ell)} \}$ is a basis of $A_n^{(\ell)}$.

\item[(b)]
For $\beta \in B^{(\ell)}$ and $\beta' \in B^{(\ell')}$, the product of the polynomials $g_\beta(\vec{x})$ and $g_{\beta'}(\vec{x})$ is expresses as 
$$
g_\beta(\vec{x}) g_{\beta'}(\vec{x})
= \sum_{\beta'' \in B^{(\ell+\ell')}}
N_{\beta \beta'}^{\beta''} g_{\beta''}(\vec{x}),
$$
where $N_{\beta \beta'}^{\beta''}$ is the non-negative integer introduced in Section \ref{sec:introduction}.
\end{itemize}
\end{lem}

\begin{proof}
For (a), the set $\{ x_\lambda |\ \lambda \in \Lambda_n^{(\ell)} \}$ clearly provides a basis of $A_n^{(\ell)}$. Since the transition matrix $( M_{\lambda \beta} )$ is invertible, $\{ g_\beta(\vec{x}) |\ \beta \in B_n^{(\ell)} \}$ gives rise to a basis of $A_n^{(\ell)}$ as well. Then, (b) follows from the obvious formula
$$
\ch(k_1, \cdots, k_\ell | \vec{x}) \ch(k'_1, \cdots, k'_{\ell'} | \vec{x})
=
\ch(k_1, \cdots, k_\ell, k'_1, \cdots, k'_{\ell'} | \vec{x}),
$$
together with the definition of $g_\beta(\vec{x})$ and that of $N_{\beta \beta'}^{\beta''}$.
\end{proof}

For $i \ge 0$ and $n \ge \ell > 1$, we define $B_n^{(\ell)}(i)$ to be the following subset in $B_n^{(\ell)}$:
$$
B_n^{(\ell)}(i)
=
\left\{
\beta = (\beta_1, \beta_2, \cdots, \beta_\ell) \in \Z^\ell \bigg| \
\begin{array}{c}
\beta_1 = i, \ \beta_2, \cdots, \beta_{\ell-1} \ge 0, \ \beta_{\ell} \ge 1 \\
\beta_1 + 2 \beta_2 + \cdots + \ell \beta_{\ell} = n
\end{array}
\right\}.
$$
In the case of $\ell = 0$ and $\ell = 1$, we also define $B_n^{(\ell)}(i)$ by
\begin{align*}
B_n^{(0)}(i) &= 
\left\{
\begin{array}{cc}
B^{(0)}, & (n = i = 0), \\
\emptyset, & \mbox{otherwise},
\end{array}
\right. &
B_n^{(1)}(i) &= 
\left\{
\begin{array}{cc}
\{ \beta = (i+1) \}, & (n = i+1), \\
\emptyset, & (n \neq i+1).
\end{array}
\right. 
\end{align*}

\begin{lem} \label{lem:derivation}
For any $n \ge \ell \ge 1$ and $\beta = (\beta_1, \cdots, \beta_\ell) \in B_n^{(\ell)}(i)$, we have:
$$
d g_{(\beta_1, \cdots, \beta_\ell)}(\vec{x}) = 
\left\{
\begin{array}{cl}
0, & (i = 0), \\
g_{(\beta_1 - 1, \beta_2, \cdots, \beta_\ell)}(\vec{x}), & (i > 0).
\end{array}
\right.
$$
\end{lem}

\begin{proof}
By the defining formula of $\ch(k_1, \cdots, k_\ell | \vec{x})_n$, we have
$$
d \ch(k_1, \cdots, k_\ell | \vec{x})_n 
= e_1(k_1, \cdots, k_\ell) \ch(k_1, \cdots, k_\ell | \vec{x})_{n-1}.
$$
This formula and the definition of $g_\beta(\vec{x})$ establish the lemma.
\end{proof}


\begin{proof}[The proof of Theorem \ref{thm:main}] 
By Lemma \ref{lem:polynomial_ring} and \ref{lem:derivation}, the subspace $J = \mathrm{Ker}(d) \subset A$ of invariant polynomials has the additive basis $\{ g_\beta(\vec{x}) |\ \beta \in B(0) \}$. Since $J \subset A$ is also a subring, we use Lemma \ref{lem:polynomial_ring} again to see that: for $\beta \in B^{(\ell)}(0)$ and $\beta' \in B^{(\ell')}(0)$, the product of the polynomials $g_\beta(\vec{x})$ and $g_{\beta'}(\vec{x})$ is expresses as 
$$
g_\beta(\vec{x}) g_{\beta'}(\vec{x})
= \sum_{\beta'' \in B^{(\ell+\ell')}}
N_{\beta \beta'}^{\beta''} g_{\beta''}(\vec{x})
= \sum_{\beta'' \in B^{(\ell+\ell')}(0)}
N_{\beta \beta'}^{\beta''} g_{\beta''}(\vec{x}).
$$
Thus, $\beta \mapsto g_\beta(\vec{x})$ induces a ring isomorphism $J' \cong J$.
\end{proof}


\subsection{Lifts to twisted cohomology}

In \cite{A-Se2}, the Chern class corresponding to $f \in J$ of positive degree is, at the first stage, constructed as a natural map 
$$
c_f : \ K_\tau(X) \longrightarrow H^*(X).
$$
Then, at the second stage, $c_f$ is lifted to the twisted cohomology
$$
C_f : \ K_\tau(X) \longrightarrow H_\eta^*(X)
$$
so that its leading term agrees with $c_f$. In the universal setting, such a lift is in one to one correspondence with a formal power series in $x_1, x_2, \cdots$:
$$
F(x_1, x_2, \cdots)
=
f(x_1, x_2, \cdots) + \mbox{higher degree term}
$$
satisfying $dF = F$. Clearly, a polynomial $f \in J$ admits various lifts $F$. A way to construct a lift \cite{A-Se2} is as follows: let $\delta : A \to A$ be the derivation defined by $\delta x_i = i x_{i+1}$. Then any $f \in J_n$ has the following series as its lift:
$$
\left(\exp \frac{\delta}{n} \right)f
=
f + \frac{1}{n} \delta f + \frac{1}{2 n^2} \delta^2f + \cdots +
\frac{1}{(i!) n^i} \delta^i f + \cdots.
$$
Using the basis $\{ g_\beta(\vec{x}) \}$, we provide another way to construct a lift:

\begin{prop}
Let $n$ and $\ell$ be such that $n \ge \ell \ge 1$. For $\beta = (\beta_1, \beta_2, \cdots, \beta_\ell) \in B^{(\ell)}_n(0)$, the polynomial $g_\beta(\vec{x}) \in J_n^{(\ell)}$ has the following series as its lift:
$$
\tilde{g}_\beta(\vec{x})
=
\sum_{i \ge 0} g_{\beta(i)}(\vec{x})
=
g_\beta(\vec{x}) + g_{\beta(1)}(\vec{x}) + g_{\beta(2)}(\vec{x}) + \cdots,
$$
where $\beta(i) = (\beta_1 + i, \beta_2, \cdots, \beta_\ell) \in B^{(\ell)}_{n+i}(i)$ for $i \ge 0$.
\end{prop}

\begin{proof}
This is an immediate consequence of Lemma \ref{lem:derivation}.
\end{proof}

By Theorem  \ref{thm:main}, the lifts $\tilde{g}_\beta(\vec{x})$ form a basis of the universal Chern classes in twisted cohomology, which answers the problem to find a basis \cite{A-Se2}.

\medskip

For $g_{(1)}(\vec{x}) = x_1 \in J_1$, our lift agrees with that of Atiyah and Segal:
$$
(\exp \delta)(g_{(1)}) = \tilde{g}_{(1)}(\vec{x})
=
\ch(1 | \vec{x})
= x_1 + x_2 + x_3 + \cdots,
$$
but not in general, as is seen in the case of $g_{(0, 1)}(\vec{x}) = x_1^2 \in J_2^{(2)}$:
\begin{align*}
\left( \exp \frac{\delta}{2} \right)x_1^2 &=
x_1^2 + x_1x_2 + \frac{1}{4}(x_2^2 + 2x_1x_3) 
+ \frac{1}{8}(x_2x_3 + x_1x_4)
+ \cdots, \\
\tilde{g}_{(0, 1)}(\vec{x}) &=
x_1^2 + x_1x_2 + x_1x_3 + x_1x_4 + \cdots.
\end{align*}

\medskip

A motivation of Atiyah and Segal to introduce Adams operations is to construct a characteristic class of twisted $K$-theory living in twisted cohomology: If $k_1, \cdots, k_\ell$ are integers such that $k_1 + \cdots + k_\ell = 1$, then the Chern character and the Adams operations combine to give a characteristic class
$$
\ch \circ \psi^{(k_1, \cdots, k_\ell)} : \ 
K_\tau(X) \longrightarrow H^*_\eta(X).
$$
In the universal setting, this corresponds to the series $\ch(k_1, \cdots, k_\ell | \vec{x})$, where $k_1, \cdots, k_\ell$ are now regarded as  numbers, rather than formal variables. With respect to our basis, we can express the series as:
$$
\ch(k_1, \cdots, k_\ell | \vec{x}) = \sum_{\beta \in B^{(\ell)}(0)}
e^\beta(k_1, \cdots, k_\ell) \tilde{g}_\beta(\vec{x}).
$$




\appendix

\section{Lists}

\subsection{Poincar\'e polynomial}

The generating function $\mathcal{J}(u,t) = \sum_{n, \ell}\mathrm{dim}J_n^{(\ell)}u^\ell t^n$ for the dimension of $J_n^{(\ell)} = A_n^{(\ell)} \cap J$ has the following formula:

$$
\mathcal{J}(u,t) 
= \frac{1-t}{(1-ut)(1-ut^2)(1-ut^3)(1-ut^4) \cdots} + t.
$$
If we substitute $u = 1$, we get the formula of the generating function $\mathcal{J}(t) = \sum_n(\mathrm{dim}J_n)t^n$ for the dimension of $J_n$ in \cite{A-Se2}:
\begin{align*}
\mathcal{J}(t)
&=
\frac{1}{(1-t^2)(1-t^3)(1-t^4) \cdots} + t \\
&=
1 + t + t^2 + t^3 + 2t^4 + 2t^5 + 4t^6 + 4t^7 + 7t^8 + 8t^9 + 12t^{10} \\
& \quad 
+ 14t^{11} + 21t^{12} + 24t^{13} + 34t^{14} 
+ 41t^{15} + 55t^{16} + 66t^{17} + 88 t^{18} \\
& \quad 
+ 105 t^{19} + 137 t^{20} + 165 t^{21} 
+ 210 t^{22} + 235 t^{23} + 320 t^{24} + \cdots.
\end{align*}
For $\ell \ge 0$, the generating function $\mathcal{J}^{(\ell)}(t) = \sum_n (\mathrm{dim} J_n^{(\ell)}) t^n$ is
\begin{align*}
\mathcal{J}^{(0)}(t) &= 1, &
\mathcal{J}^{(1)}(t) &= t, &
\mathcal{J}^{(\ell)}(t)
&=
\frac{t^\ell}{(1-t^2) (1 - t^3) \cdots (1- t^\ell)}. \ (\ell \ge 2).
\end{align*}
A calculation gives the table:

\begin{center}
\begin{tabular}{c|cccccccccccccccc|c}
$\ell$ &
1 & 2 & 3 & 4 & 5 & 6 & 7 & 8 & 9 & 10 & 11 & 12 & 13 & 14 & 15 & 16 &
total \\
\hline
$\mathrm{dim} J_1^{(\ell)}$ &
1 &  &  &  &  &  &  &  & & & & & & & & & 
1 \\
\hline
$\mathrm{dim} J_2^{(\ell)}$ &
0 & 1 &  &  &  &  &  &  & & & & & & & & & 
1 \\
\hline
$\mathrm{dim} J_3^{(\ell)}$ &
0 & 0 & 1 &  &  &  &  &  & & & & & & & & & 
1 \\
\hline
$\mathrm{dim} J_4^{(\ell)}$ &
0 & 1 & 0 & 1 &  &  &  &  & & & & & & & & & 
2 \\
\hline
$\mathrm{dim} J_5^{(\ell)}$ &
0 & 0 & 1 & 0 & 1 &  &  &  & &  & & & & & & &
2 \\
\hline
$\mathrm{dim} J_6^{(\ell)}$ &
0 & 1 & 1 & 1 & 0 & 1 &  &  & & & & & & & &  &
4 \\
\hline
$\mathrm{dim} J_7^{(\ell)}$ &
0 & 0 & 1 & 1 & 1 & 0 & 1 &  & & & & & & & & &
4 \\
\hline
$\mathrm{dim} J_8^{(\ell)}$ & 
0 & 1 & 1 & 2 & 1 & 1 & 0 & 1 & & & & & & & & &
7 \\
\hline
$\mathrm{dim} J_9^{(\ell)}$ &
0 & 0 & 2 & 1 & 2 & 1 & 1 & 0 & 1 & & & & & & & &
8 \\
\hline
$\mathrm{dim} J_{10}^{(\ell)}$ &
0 & 1 & 1 & 3 & 2 & 2 & 1 & 1 & 0 & 1 & & & & & & &
12 \\
\hline
$\mathrm{dim} J_{11}^{(\ell)}$ &
0 & 0 & 2 & 2 & 3 & 2 & 2 & 1 & 1 & 0 & 1 & & & & & &
14 \\
\hline
$\mathrm{dim} J_{12}^{(\ell)}$ &
0 & 1 & 2 & 4 & 3 & 4 & 2 & 2 & 1 & 1 & 0 & 1 & & & & & 21
\\
\hline
$\mathrm{dim} J_{13}^{(\ell)}$ &
0 & 0 & 2 & 3 & 
5 & 3 & 4 & 2 & 2 & 
1 & 1 & 0 & 1 & & & & 
24 \\
\hline
$\mathrm{dim} J_{14}^{(\ell)}$ &
0 & 1 & 2 & 5 & 
5 & 6 & 4 & 4 & 2 & 
2 & 1 & 1 & 0 & 1 & & & 
34 \\
\hline
$\mathrm{dim} J_{15}^{(\ell)}$ &
0 & 0 & 3 & 4 & 
7 & 6 & 6 & 4 & 4 & 
2 & 2 & 1 & 1 & 0 & 1 & & 
41 \\
\hline
$\mathrm{dim} J_{16}^{(\ell)}$ &
0 & 1 & 2 & 7 & 7 & 
9 & 7 & 7 & 4 & 4 & 
2 & 2 & 1 & 1 & 0 & 
1 & 
55 \\
\end{tabular}
\end{center}


\subsection{Lists of bases}

In the following, the monomials $x_\lambda$ in the invariant polynomial $g_\beta(\vec{x})$ are arranged by using the ordering $L'_n$ on the set of partitions \cite{M}.

\subsubsection{$\ell = 1$ and $\ell = 2$}

We have
\begin{align*}
\mathcal{J}^{(1)}(t) &= t, &
\mathcal{J}^{(2)}(t) &= t^2 + t^4 + t^6 + t^8 + \cdots
\end{align*}

The base in $J^{(1)}$ is $g_{(1)} = x_1$ and the bases in $J^{(2)}$ are
\begin{align*}
g_{(0,1)}
&=
x_1^2 \\
g_{(0,2)}
&=
x_2^2 - 2 x_1x_3 \\
g_{(0,3)}
&=
x_3^2 - 2 x_2x_4 + 2 x_1x_5 \\
g_{(0,4)}
&=
x_4^2 - 2 x_3x_5 + 2 x_2x_6 - 2x_1x_7 \\
g_{(0,5)}
&=
x_5^2 - 2 x_4x_6 + 2 x_3x_7 - 2x_2x_8 + 2 x_1x_9 \\
 & \vdots \\
g_{(0, m)}
&=
x_m^2 - 2 x_{m-1}x_{m+1} + 2 x_{m-2}x_{m+2} - \cdots
+ (-1)^{m-1} 2 x_1x_{2m-1}
\end{align*}


\subsubsection{$\ell = 3$}

We have
$$
\mathcal{J}^{(3)}(t) = t^3 + t^5 + t^6 + t^7 + t^8 + 2t^9 + 
t^{10} + 2t^{11} + 2t^{12} + 2t^{13} + 2t^{14} + 3t^{15} + \cdots
$$

\begin{align*}
g_{(0,0,1)}
&=
x_1^3 \\
g_{(0,1,1)}
&=
x_1(x_2^2 - 2x_1x_3) \\
g_{(0,0,2)}
&=
x_2^3 - 3 x_1x_2x_3 + 3 x_1^2x_4 \\
g_{(0,2,1)}
&=
x_1(x_3^2 - 2x_2x_4 + 2x_1x_5) \\
g_{(0,1,2)}
&=
x_2x_3^2 - 2 x_2^2x_4 - x_1x_3x_4 + 5 x_1x_2x_5 - 5 x_1^2x_6 \\
g_{(0,0,3)}
&=
x_3^3 - 3 x_2x_3x_4 + 3 x_2^2x_5 
+ 3 x_1x_4^2 - 3 x_1x_3x_5 - 3 x_1x_2x_6 + 3 x_1^2x_7 \\
g_{(0,3,1)}
&=
x_1(x_4^2 - 2 x_3x_5 + 2 x_2x_6 - 2 x_1x_7) \\
g_{(0,2,2)} 
&=
x_2x_4^2 - 2 x_2x_3x_5 + 2 x_2^2x_6 - x_1x_4x_5 + 3 x_1x_3x_6 
- 7 x_1x_2x_7 + 7 x_1^2x_8 \\
g_{(0,1,3)}
&=
x_3x_4^2 - 2 x_3^2x_5 - x_2x_4x_5  + 5 x_2x_3x_6 - 5 x_2^2x_7 \\
&\quad 
+ 4 x_1x_5^2 - 7 x_1x_4x_6 + 2 x_1x_3x_7 + 8 x_1x_2x_8 - 8 x_1^2x_9 \\
g_{(0,4,1)}
&=
x_1(x_5^2 - 2 x_4x_6 + 2 x_3x_7 - 2 x_2x_8 + 2 x_1x_9) \\
 & \\
g_{(0,3,2)}
&=
x_2x_5^2 - 2 x_2x_4x_6 +  2 x_2x_3x_7 - 2 x_2^2x_8 \\
&\quad
- x_1x_5x_6 + 3 x_1x_4x_7  - 5 x_1x_3x_8 + 9 x_1x_2x_9 -9 x_1^2 x_{10} \\
g_{(0,0,4)}
&=
x_4^3 - 3 x_3x_4x_5 + 3 x_2x_5^2 + 3 x_3^2x_6 
- 3 x_2x_4x_6  - 3 x_2x_3x_7  + 3 x_2^2x_8 \\
&\quad
- 3 x_1x_5x_6 + 6 x_1x_4x_7 - 3 x_1x_3x_8 - 3 x_1x_2x_9 + 3 x_1^2x_{10} \\
 & \\
g_{(0,2,3)}
&=
x_3 x_5^2 - 2 x_3 x_4 x_6 + 2 x_3^2 x_7 
- x_2 x_5 x_6 + 3 x_2 x_4 x_7 - 7 x_2 x_3 x_8 + 7 x_2^2 x_9 \\
&\quad
+ 5 x_1x_6^2 - 9 x_1x_5x_7 + 6 x_1x_4x_8 + x_1x_3x_9
- 15 x_1x_2x_{10} + 15 x_1^2x_{11} \\
g_{(0,5,1)}
&=
x_1(x_6^2 - 2 x_5 x_7 +  2 x_4 x_8 - 2 x_3 x_9
+ 2 x_2x_{10} - 2 x_1x_{11}) \\
 & \\
g_{(0,4,2)}
&=
x_2x_6^2 - 2 x_2x_5x_7 + 2 x_2x_4x_8 - 2 x_2x_3x_9 + 2 x_2^2x_{10} \\
&\quad
- x_1 x_6 x_7 
+ 3 x_1 x_5 x_8 
-   5 x_1 x_4 x_9
+ 7 x_1 x_3 x_{10}  
- 11 x_1 x_2 x_{11} 
+ 11 x_1^2 x_{12} \\
g_{(0,1,4)}
&=
x_4x_5^2 - 2 x_4^2x_6 - x_3x_5x_6 + 5 x_3x_4x_7 - 5 x_3^2x_8 \\
&\quad 
+ 4 x_2x_6^2 - 7 x_2x_5x_7 + 2 x_2x_4x_8 + 8 x_2x_3x_9 - 8 x_2^2x_{10} \\
& \quad
- 4 x_1x_6x_7 
+ 11 x_1x_5x_8 
- 13 x_1x_4x_9
+ 5 x_1x_3x_{10}  
+ 11 x_1x_2x_{11}  
-11 x_1^2x_{12} \\
 & \\
g_{(0,0,5)}
&=
x_5^3 
- 3 x_4 x_5 x_6 + 3 x_4^2 x_7 
+ 3 x_3 x_6^2 - 3 x_3 x_5 x_7 - 3 x_3 x_4 x_8 + 3 x_3^2 x_9 \\
&\quad
- 3 x_2 x_6 x_7 
+ 6 x_2 x_5 x_8 
- 3 x_2 x_4 x_9 
- 3 x_2 x_3 x_{10}
+ 3 x_2^2 x_{11} 
+ 3 x_1 x_7^2 \\
&\quad
- 3 x_1 x_6 x_8 
- 3 x_1 x_5 x_9
+ 6 x_1 x_4 x_{10} 
- 3 x_1 x_3 x_{11} 
- 3 x_1 x_2 x_{12}  
+ 3 x_1^2 x_{13} \\
g_{(0,3,3)}
&=
x_3 x_6^2 
- 2 x_3 x_5 x_7 
+ 2 x_3 x_4 x_8 
- 2 x_3^2 x_9 
- x_2 x_6 x_7 \\
&\quad
+ 3 x_2 x_5 x_8 
- 5 x_2 x_4 x_9 
+ 9 x_2 x_3 x_{10} 
- 9 x_2^2 x_{11} 
+ 6 x_1 x_7^2 
- 11 x_1 x_6 x_8 \\
&\quad
+ 8 x_1 x_5 x_9
- 3 x_1 x_4 x_{10}  
- 6 x_1 x_3 x_{11}  
+ 24 x_1 x_2 x_{12}  
-24 x_1^2 x_{13} \\
g_{(0,6,1)}
&=
x_1 (x_7^2 
- 2 x_6 x_8 
+ 2 x_5 x_9
- 2 x_4 x_{10}  
+ 2 x_3 x_{11}  
- 2 x_2 x_{12}  
+ 2 x_1 x_{13})
\end{align*}


\subsubsection{$\ell = 4$}

We have
$$
\mathcal{J}^{(4)}(t) = t^4 + t^6 + t^7 + 2t^8 + t^9 + 
3t^{10} + 2t^{11} + 4t^{12}  + 
3 t^{13} + 5 t^{14} + 4 t^{15} + \cdots
$$

\begin{align*}
g_{(0,0,0,1)}
&=
x_1^4 \\
g_{(0,1,0,1)}
&=
x_1^2(x_2^2 - 2 x_1x_3) \\
g_{(0,0,1,1)}
&=
x_1(x_2^3 - 3 x_1x_2x_3 + 3 x_1^2x_4) \\
g_{(0,2,0,1)}
&=
x_1^2(x_3^2 - 2 x_2x_4 + 2 x_1x_5) \\
g_{(0,0,0,2)}
&=
x_2^4 - 4 x_1x_2^2x_3 + 2x_1^2x_3^2 + 4x_1^2x_2x_4 - 4x_1^3x_5 \\
g_{(0,1,1,1)}
&=
x_1(x_2x_3^2 - 2 x_2^2x_4 - x_1x_3x_4 + 5 x_1x_2x_5 - 5 x_1^2x_6) \\
 & \\
g_{(0,0,2,1)}
&=
x_1(x_3^3 - 3 x_2 x_3 x_4 + 3 x_2^2 x_5+ 3 x_1 x_4^2
- 3 x_1 x_3 x_5 - 3 x_1 x_2 x_6 + 3 x_1^2 x_7) \\
g_{(0,1,0,2)}
&=
x_2^2 x_3^2  - 2 x_2^3 x_4 
- 2 x_1 x_3^3
+ 4 x_1 x_2 x_3 x_4 
+ 2 x_1 x_2^2 x_5 \\
&\quad
- 3 x_1^2 x_4^2 + 2 x_1^2 x_3 x_5 - 6 x_1^2 x_2 x_6 + 6 x_1^3 x_7 \\
g_{(0,3,0,1)}
&=
x_1^2(x_4^2 - 2 x_3 x_5 + 2 x_2 x_6 - 2 x_1 x_7) \\
 & \\
g_{(0,0,1,2)}
&=
x_2 x_3^3 - 3 x_2^2 x_3 x_4 + 3 x_2^3 x_5 
- x_1 x_3^2 x_4 
+ 5 x_1 x_2 x_4^2
- 2 x_1 x_2 x_3 x_5
- 7 x_1 x_2^2 x_6 \\
&\quad
- 5 x_1^2 x_4 x_5
+ 7 x_1^2 x_3 x_6
+ 7 x_1^2 x_2 x_7 
- 7 x_1^3 x_8 \\
g_{(0,2,1,1)}
&=
x_1(x_2 x_4^2 - 2 x_2 x_3 x_5 + 2 x_2^2 x_6 
- x_1 x_4 x_5 
+ 3 x_1 x_3 x_6 
- 7 x_1 x_2 x_7 
+ 7 x_1^2 x_8) \\
 & \\
g_{(0,0,0,3)}
&=
x_3^4 - 4 x_2x_3^2x_4 + 2 x_2^2x_4^2  + 4 x_2^2x_3x_5 - 4 x_2^3x_6 \\
&\quad
+ 4 x_1 x_3 x_4^2 - 4 x_1 x_3^2 x_5
- 8 x_1 x_2 x_4 x_5 + 8 x_1 x_2 x_3 x_6 + 4 x_1 x_2^2 x_7 \\
&\quad
+ 6 x_1^2 x_5^2
- 4 x_1^2 x_4 x_6
- 4 x_1^2 x_3 x_7
- 4 x_1^2 x_2 x_8 
+ 4 x_1^3 x_9 \\
g_{(0,2,0,2)}
&=
x_2^2 x_4^2  - 2 x_2^2 x_3 x_5   + 2 x_2^3 x_6 \\
&\quad
- 2 x_1 x_3 x_4^2
+ 4 x_1 x_3^2 x_5
- 4 x_1 x_2 x_3 x_6  
- 2 x_1 x_2^2 x_7 \\
&\quad
- 4 x_1^2 x_5^2
+ 8 x_1^2 x_4 x_6
- 4 x_1^2 x_3 x_7
+ 8 x_1^2 x_2 x_8 
- 8 x_1^3 x_9 \\
g_{(0,1,2,1)}
&=
x_1 \left(
x_3x_4^2 - 2 x_3^2x_5 - x_2x_4x_5  + 5 x_2x_3x_6 - 5 x_2^2x_7 \right. \\
&\quad \left.
+ 4 x_1x_5^2 - 7 x_1x_4x_6 + 2 x_1x_3x_7 + 8 x_1x_2x_8 - 8 x_1^2x_9 \right) \\
g_{(0,4, 0,1)}
&=
x_1^2(x_5^2 - 2 x_4 x_6 + 2 x_3 x_7 - 2 x_2 x_8 + 2 x_1 x_9)
\end{align*}

\subsubsection{$\ell \ge 5$}


$$
\mathcal{J}^{(5)}(t)
=
t^5 + t^7 + t^8 + 2 t^9 + 2 t^{10} + 3 t^{11} + 3 t^{12} 
+ 5 t^{13} + 5 t^{14} + 7 t^{15} + \cdots.
$$
\begin{align*}
g_{(0,0,0,0,1)}
&=
x_1^5 \\
g_{(0,1,0,0,1)}
&=
x_1^3(x_2^2 - 2 x_1x_3) \\
g_{(0,0,1,0,1)}
&=
x_1^2(x_2^3 - 3 x_1x_2x_3 + 3 x_1^2x_4) \\
 & \\
g_{(0,0,0,1,1)}
&=
x_1(x_2^4 - 4 x_1x_2^2x_3 + 2 x_1^2x_3^2 + 4 x_1^2x_2x_4 - 4 x_1^3x_5) \\
g_{(0,2,0,0,1)}
&=
x_1^3(x_3^2 - 2 x_2x_4 + 2 x_1x_5) \\
 & \\
g_{(0,0,0,0,2)}
&=
x_2^5 - 5 x_1 x_2^3 x_3 + 5 x_1^2 x_2 x_3^2 + 5 x_1^2 x_2^2 x_4 
- 5 x_1^3 x_3 x_4 - 5 x_1^3 x_2 x_5 + 5 x_1^4 x_6 \\
g_{(0,1,1,0,1)}
&=
x_1^2(x_2x_3^2 - 2 x_2^2x_4 - x_1x_3x_4 + 5 x_1x_2x_5 - 5 x_1^2x_6) \\
 & \\
g_{(0,1,0,1,1)}
&=
x_1(x_2^2 x_3^2 - 2 x_2^3 x_4
- 2 x_1 x_3^3
+ 4 x_1 x_2 x_3 x_4
+ 2 x_1 x_2^2 x_5 \\
&\quad
- 3 x_1^2 x_4^2
+ 2 x_1^2 x_3 x_5
- 6 x_1^2 x_2 x_6
+ 6 x_1^3 x_7) \\
g_{(0,0,2,0,1)}
&=
x_1^2(x_3^3 - 3 x_2 x_3 x_4 + 3 x_2^2 x_5 + 3 x_1 x_4^2 
- 3 x_1 x_3 x_5 - 3 x_1 x_2 x_6 + 3 x_1^2 x_7) \\
g_{(0,3,0,0,1)}
&=
x_1^3 (x_4^2 - 2 x_3 x_5 + 2 x_2 x_6 - 2 x_1 x_7)
\end{align*}

$$
\mathcal{J}^{(6)}(t)
=
t^6 + t^8 + t^9 + 2 t^{10} + 2 t^{11} + 4 t^{12} + 
3 t^{13} + 6 t^{14} + 6 t^{15} + \cdots.
$$
\begin{align*}
g_{(0,0,0,0,0,1)}
&=
x_1^6 \\
g_{(0,1,0,0,0,1)}
&=
x_1^4(x_2^2 - 2 x_1x_3) \\
g_{(0,0,1,0,0,1)}
&=
x_1^3(x_2^3 - 3 x_1x_2x_3 + 3 x_1^2x_4) \\
 & \\
g_{(0,0,0,1,0,1)}
&=
x_1^2(x_2^4 - 4 x_1x_2^2x_3 + 2 x_1^2x_3^2 + 4 x_1^2x_2x_4 - 4 x_1^3x_5) \\
g_{(0,2,0,0,0,1)}
&=
x_1^4 (x_3^2 - 2 x_2 x_4 + 2 x_1 x_5) \\
 & \\
g_{(0,0,0,0,1,1)}
&=
x_1(x_2^5 - 5 x_1x_2^3x_3 + 5 x_1^2x_2x_3^2 + 5 x_1^2x_2^2x_4 \\
&\quad 
- 5 x_1^3 x_3 x_4 - 5 x_1^3 x_2 x_5 + 5 x_1^4 x_6) \\
g_{(0,1,1,0,0,1)}
&=
x_1^3(x_2 x_3^2 - 2 x_2^2x_4 - x_1x_3x_4 + 5 x_1x_2x_5 - 5 x_1^2x_6)
\end{align*}

$$
\mathcal{J}^{(7)}(t)
=
t^7 + t^9 + t^{10} + 2 t^{11} + 2 t^{12} 
+ 4 t^{13} + 4 t^{14} + 6 t^{15} + \cdots
$$

\begin{align*}
g_{(0,0,0,0,0,0,1)}
&=
x_1^7 \\
g_{(0,1,0,0,0,0,1)}
&=
x_1^5(x_2^2 - 2 x_1x_3) \\
g_{(0,0,1,0,0,0,1)}
&=
x_1^4(x_2^3 - 3 x_1x_2x_3 + 3 x_1^2x_4)
\end{align*}



\end{document}